\documentclass[11pt]{article}

\usepackage[margin=1.05in]{geometry}
\usepackage{amsmath,amssymb,amsthm,mathtools}
\usepackage{booktabs}
\usepackage{array}
\usepackage{enumitem}
\usepackage{hyperref}
\usepackage{xcolor}
\usepackage{longtable}

\hypersetup{colorlinks=true,linkcolor=blue,citecolor=blue,urlcolor=blue}

\newtheorem{theorem}{Theorem}[section]
\newtheorem{lemma}[theorem]{Lemma}

\newtheorem{corollary}[theorem]{Corollary}
\newtheorem{question}[theorem]{Question}
\newtheorem{conjecture}[theorem]{Conjecture}
\theoremstyle{definition}

\newtheorem{example}[theorem]{Example}

\newcommand{\R}{\mathbb{R}}
\newcommand{\Z}{\mathbb{Z}}

\newcommand{\cL}{\mathcal{L}}

\newcommand{\den}{\operatorname{den}}

\begin{document}
\begingroup
\renewcommand{\thefootnote}{\fnsymbol{footnote}}

\begin{center}
{\huge Resistance Curvature: Recognition, Polyhedral Structure, and Graph Products}\\[18pt]
{\Large Meng Guo$^{1}$, Wensheng Sun$^{2}$, and Yujun Yang$^{3,1}$\footnotemark[1]}\\[6pt]
{\footnotesize
$^{1}$ School of Mathematics and Information Sciences, Yantai University, Yantai, Shandong 264005, China\\
$^{2}$ School of Mathematics and Statistics, Gansu Center for Applied Mathematics, Lanzhou University, Lanzhou, Gansu 730000, China\\
$^{3}$ Department of Mathematics and Artificial Intelligence, Qilu University of Technology (Shandong Academy of Sciences), Jinan, Shandong 250100, China\\
E-mail addresses: \texttt{mg060328@163.com}, \texttt{wensheng07002@163.com}, and \texttt{yangyujun@qlu.edu.cn}}
\end{center}

\footnotetext[1]{Corresponding author. E-mail address: \texttt{yangyujun@qlu.edu.cn}.}
\endgroup

\vspace{1mm}
\begin{abstract}
		Resistance curvature, introduced by Devriendt and Lambiotte, is a novel discrete curvature notion defined through effective resistance on weighted graphs. A graph is called resistance nonnegative if there exists a choice of positive edge weights for which the resistance curvature is nonnegative at every vertex. This property has a notable combinatorial interpretation in terms of random spanning trees: a graph is resistance nonnegative if and only if it admits a distribution on its spanning trees under which every vertex has expected degree at most two. Equivalently, resistance nonnegativity can be characterized by the intersection of the relative interior of the spanning tree polytope with the twice-dilated matching polytope. These characterizations reveal strong connections among resistance curvature, effective resistance, spanning tree distributions, matching theory, and polyhedral combinatorics. Based on the sign of the curvature, Devriendt introduced the classes of resistance nonnegative (RN), resistance positive (RP), and strictly resistance nonnegative (SRN) graphs, and posed several questions concerning their recognition, polyhedral structure, and structural properties.
		
		In this paper, we first answer Devriendt's question on the computational complexity of recognizing RN, RP, and SRN graphs by proving that all three classes can be recognized in polynomial time. We then address his question concerning the tree double matching polytope $\Theta(G)=P(G)\cap 2M(G)$, where $P(G)$ and $M(G)$ denote the spanning tree polytope and the matching polytope of $G$, respectively.  Further, we characterize the vertices of $\Theta(G)$ in terms of full-rank systems of tight constraints. Whenever $\Theta(G)\neq\emptyset$, we also determine the least positive integer $k_G$ such that $k_G\Theta(G)$ is a lattice polytope. Finally, for every finite Cartesian product of paths, we explicitly construct an average point satisfying the condition for resistance nonnegativity, thereby obtaining that such graphs are RN. We further characterize the classes of such Cartesian product graphs that are RP or SRN.
	\end{abstract}
	
     \textbf{Keywords:}
		resistance curvature; spanning tree polytope; matching polytope; polynomial time; Cartesian product; Hamiltonian path.
        
  {\bf 2020 Mathematics Subject Classification:} 05C12, 05B35, 05C70, 90C27.
	
	\section{Introduction}
	\subsection{Resistance curvature}

Throughout the paper, all graphs are finite, simple, and connected unless stated otherwise. Let $G=(V,E)$, with $n=|V|$ and $m=|E|$. A positive conductance function on $G$ is a map $c:E\to\mathbb R_{>0}$. The weighted graph $(G,c)$ can naturally be regarded as an electrical network. In this associated electrical network, each edge $e\in E$ is viewed as a resistor with resistance $c_e^{-1}$. For distinct vertices $u,v\in V$, the effective resistance $\omega_{uv}(c)$ is the potential difference between $u$ and $v$ resulting from injecting one unit of current at $u$ and extracting it at $v$. As a fundamental concept in circuit theory, effective resistance has been extensively studied since Kirchhoff's current law~\cite{KirchhoffKL,SeshuReedLG}, and its computation remains a classical problem in physics and engineering. Although effective resistance originated in the study of electrical circuits and potential theory, it has a purely mathematical interpretation. Its equivalent definitions can be given from combinatorial, algebraic, and geometric perspectives; see, for example,~\cite{BapatGutmanXiaoRD,DevriendtER,KleinRandicRD,XiaoGutmanRD}.
	
A \emph{spanning tree} of $G$ is a connected acyclic spanning subgraph of $G$, and we denote the set of all spanning trees of $G$ by $\mathcal T(G)$. For each edge $uv\in E$, the effective resistance $\omega_{uv}(c)$ admits the following representation in terms of weighted spanning trees:
\begin{equation}\label{eq:effective_resistance}
\omega_{uv}(c)=\frac{1}{c_{uv}}\frac{\displaystyle\sum_{T\in\mathcal T(G):\,uv\in E(T)}\prod_{e\in E(T)}c_e}{\displaystyle\sum_{T\in\mathcal T(G)}\prod_{e\in E(T)}c_e}.
\end{equation}
	
Following Devriendt~\cite{DevriendtRN}, for each edge $e=uv\in E$, the \emph{relative resistance} of $uv$ is defined by $r_{uv}(c):=c_{uv}\omega_{uv}(c)$. By~\eqref{eq:effective_resistance}, we have
\[
r_{uv}(c)=\frac{\displaystyle\sum_{T\in\mathcal T(G):\,uv\in E(T)}\prod_{e\in E(T)}c_e}{\displaystyle\sum_{T\in\mathcal T(G)}\prod_{e\in E(T)}c_e}=\mathbb P_c\bigl(uv\in E(T)\bigr),
\]
where $T$ is sampled from the $c$-weighted distribution on $\mathcal T(G)$ in which the probability of a spanning tree is proportional to the product of the conductances of its edges. Thus, $r_{uv}(c)$ is precisely the probability that the edge $uv$ belongs to a $c$-weighted random spanning tree.
	
	Motivated by the geometry of effective resistance, Devriendt and Lambiotte \cite{DevriendtLambiotteDC} introduced a discrete scalar curvature on the vertices of a weighted graph. More precisely, the \emph{resistance curvature} of $(G, c)$ at a vertex $v \in V$ is defined by
\begin{equation}\label{eq:curvature}
p_v(c)=1-\frac{1}{2}\sum_{u\sim v}c_{uv}\omega_{uv}(c)=1-\frac{1}{2}\sum_{u\sim v}r_{uv}(c).
\end{equation}
	Moreover, Foster's first theorem \cite{FosterAI} states that
	\begin{equation} \label{eq:foster}
		\sum_{e \in E} r_e(c) = |V| - 1
	\end{equation}
	for every positive conductance function $c$. Summing \eqref{eq:curvature} over all vertices and using \eqref{eq:foster}, we obtain \(\sum_{v \in V} p_v(c)=1\).
	This identity may be viewed as a discrete analogue of a total curvature formula.
	
	Resistance curvature provides a direct link between the geometry of electrical networks and the combinatorics of spanning trees. For a vertex $v\in V$, the sum
	\(\sum_{u \sim v} c_{uv} \, \omega_{uv}(c)\)
	is the expected degree of $v$ in such a random spanning tree $T$:
	\[
		\sum_{u\sim v}c_{uv}\omega_{uv}(c)
		=
		\sum_{u\sim v}
		\mathbb{E}_c\!\left[\mathbf{1}_{\{uv\in T\}}\right]
		=
		\mathbb{E}_c[d_T(v)].
	\]
	Hence
	\begin{equation}\label{eq:degree}
		p_v(c) = 1 - \frac{1}{2} \, \mathbb{E}_c\big[ d_T(v) \big],
	\end{equation}
	where $d_T(v)$ denotes the degree of $v$ in $T$. Thus, nonnegative resistance curvature at $v$ is equivalent to requiring that the expected degree of $v$ in the corresponding random spanning tree be at most two.
	
The preceding discussion shows that resistance curvature brings together ideas from discrete geometry, electrical network theory, random spanning trees, and combinatorial optimization. In particular, equation~\eqref{eq:degree} allows techniques from polyhedral combinatorics, matroid theory, and matching theory to be applied to the study of resistance curvature. Moreover, resistance curvature is related to other notions of discrete curvature, such as Ollivier--Ricci curvature and Forman--Ricci curvature~\cite{DevriendtLambiotteDC}. These connections not only make resistance curvature an interesting object of theoretical study, but also suggest its potential as a geometric descriptor for graph-structured data analysis and graph learning~\cite{FeiCN,SouthernCF}.
	
In this direction, Devriendt~\cite{DevriendtRN} introduced the following classification of graphs according to the sign of resistance curvature. A graph $G$ is called \emph{resistance nonnegative (RN)} if there exists a positive conductance function $c$ such that $p_v(c)\geq0$ for every $v\in V$. It is called \emph{resistance positive (RP)} if the inequalities can all be made strict, and \emph{strictly resistance nonnegative (SRN)} if it is RN but not RP. A central problem in the study of resistance curvature is to determine for which graphs $G$ there exists a choice of edge weights $c$ such that $G$ is RN, SRN or RP. From a geometric viewpoint, regarding the edge weights as inverse edge lengths allows the problem to be reformulated as determining which graphs admit a metric with nonnegative or positive resistance curvature. This is a discrete analogue of the classical problem in differential geometry of determining which manifolds admit metrics with nonnegative or positive scalar curvature \cite{SchoenYauPC}.
	\subsection{The polyhedral framework}
	Let \(G = (V, E)\) be a graph, where \(V\) and \(E\) denote the vertex set and the edge set of \(G\), respectively.
	For a vector $x = (x_e)_{e \in E} \in \mathbb{R}^E$ and a subset $F \subseteq E$, write \(x(F) = \sum_{e \in F} x_e\).
	For $U \subseteq V$, let $E(U) = \{ uv \in E : u, v \in U \}$ denote the edge set of the subgraph induced by $U$. The degree of a vector \(x\) at a vertex \(v\) is
	\(d_v(x) = \sum_{u: uv \in E} x_{uv}\). For each $T \in \mathcal{T}(G)$, let
	\[
	\chi^T(e) =
	\begin{cases}
		1 & \text{if } e \in T, \\
		0 & \text{otherwise.}
	\end{cases}
	\]
	be the incidence vector of $T$. The \emph{spanning tree polytope} of $G$ is defined as the convex hull of the incidence vectors of all spanning trees of $G$, that is,
	\begin{equation} \label{eq:tree_polytope}
		P(G) = \operatorname{conv}\{ \chi^T : T \in \mathcal{T}(G) \}.
	\end{equation}
	Equivalently, $P(G)$ consists of all vectors of the form
\(	x = \sum_{T \in \mathcal{T}(G)} \lambda_T \chi^T\),
	where $\lambda_T \geq 0$ for every $T \in \mathcal{T}(G)$ and $\sum_{T \in \mathcal{T}(G)} \lambda_T = 1$.
	
	For $F \subseteq E$, let $\kappa_G(F)$ denote the number of connected components of the spanning subgraph $(V, F)$ and put
	\[
	r_G(F) = n - \kappa_G(F),
	\]   which is the rank function of the graphic matroid of \(G\).
	An intuition behind \(r_G(F)\) is that it is the largest acyclic set of edges that can be chosen from
	$F$, that is $r_G(F) = max\{|A| : A\subseteq F, A \text{ is acyclic}\}$, so we have $ x(F)\le r_G(F)$ \cite{EdmondsGA,EdmondsSF}.
	
	Edmonds' description of the graphic matroid base polytope is
	\begin{equation} \label{eq:edmonds}
		P(G) = \left\{ x \in \mathbb{R}^E_{\geq 0} :
		x(E) = n - 1, \;
		x(F) \leq r_G(F) \text{ for every } F \subseteq E \right\}.
	\end{equation}
	Equivalently, we can use the fact that spanning trees form a matroid \cite{MohammadiNA}. Thus the rank inequalities may be  replaced by
	\[
	x(E(U)) \leq |U| - 1 \quad \text{for all nonempty proper sets } U \subseteq V.
	\]
	 Since $P(G)$ need not be full-dimensional in $\mathbb{R}^E$, we let $P(G)^\circ$ denote the relative interior of \(P(G)\) in its affine subspace.
	 	
	 We use the following characterization of RN and RP graphs from Devriendt's paper \cite[Theorem 3.8, Corollary 3.9]{DevriendtRN}:
		\begin{equation}\label{eq:RN-RP-polytope}
		\begin{aligned}
			G\text{ is RN}
			&\Longleftrightarrow
			P(G)^\circ\cap \{x:d_v(x)\le 2\text{ for every }v\in V\}\neq\emptyset,\\
			G\text{ is RP}
			&\Longleftrightarrow
			P(G)^\circ\cap \{x:d_v(x)< 2\text{ for every }v\in V\}\neq\emptyset.
		\end{aligned}
	\end{equation}
	Equivalently, RN and RP are characterized by positive probability distributions on $\mathcal{T}(G)$ whose expected vertex degrees are at most two and strictly smaller than two, respectively.

	Let
	\[
	M(G) = \operatorname{conv}\{ \chi^M : M \subseteq E \text{ is a matching} \}
	\]
	be the \emph{matching polytope} of $G$. Scaling the hyperplane description of the matching polytope by a factor of two yields the double matching polytope $2M(G)$ \cite{EdmondsMC}:
	\begin{equation} \label{eq:matching_polytope}
		2M(G) = \left\{ x \in \mathbb{R}^E_{\geq 0} :
		\begin{array}{l}
			d_v(x) \leq 2 \quad (v \in V), \\[2pt]
			x(E(U))\leqslant 2\lfloor\frac{1}{2}|U|\rfloor \text{ for all odd-sized }U\subseteq V
		\end{array}
		\right\}
	\end{equation}
	 Devriendt introduced the \emph{tree double matching (TDM) polytope}, defined by
	\(\Theta(G) = P(G) \cap 2M(G)\).
	
	If $x \in P(G)$ and $U \subseteq V$ is odd, then, by \eqref{eq:edmonds},
	\[
	x(E(U)) \leq r_G(E(U)) = |U| - \kappa(G[U]) \leq |U| - 1 = 2 \left\lfloor \frac{|U|}{2} \right\rfloor.
	\]
	Thus the odd-set inequalities are automatic on $P(G)$, and hence
	\begin{equation} \label{eq:TMD-polytope}
		\Theta(G) = P(G) \cap \{ x : d_v(x) \leq 2 \text{ for every } v \in V \}.
	\end{equation}
	That is, a point in $\Theta(G)$ only needs to lie in $P(G)$ and satisfy the degree constraints. In particular, $G$ is RN if and only if $P(G)^\circ \cap \Theta(G) \neq \emptyset$. These characterizations place resistance curvature at the intersection of effective resistance, spanning tree distributions, matching theory, and polyhedral combinatorics.
	
Motivated by the polyhedral characterization of resistance nonnegativity, Devriendt posed the following questions and conjecture:
	\begin{question}\label{que:3.11}
		\cite[Question 3.11]{DevriendtRN}
		What is the computational complexity of the decision problems:
		For a given graph $G$, decide if $G$ is RN/RP/SRN?
	\end{question}
	\begin{question}\label{que:4.10}
		\cite[Question 4.10]{DevriendtRN}
		What are the vertices of \(\Theta(G)\)? What is the smallest $k$ such that the $k$-dilated TDM polytope is a lattice polytope?
	\end{question}
	\begin{conjecture}\label{con:6.8}
		\cite[Conjecture 6.8]{DevriendtRN}
		The grid graph $P_{n} \square P_{m}$ is RN for all $m, n \in \mathbb{N}$.
	\end{conjecture}
	\begin{question}\label{que:6.9}\cite[Question 6.9]{DevriendtRN}
	Are all finite Cartesian products of paths RN?
	\end{question}
    
Clearly, Conjecture \ref{con:6.8}  is a special case of Question \ref{que:6.9}.
	\subsection{Our results}
This paper focuses on the above problems, and our main results are summarized as follows.
	
	(1) We answer Question~\ref{que:3.11} completely. Theorem~\ref{thm:complexity} proves that
	$\mathrm{RN}$, $\mathrm{RP}$, and $\mathrm{SRN}$ recognition
	are all solvable in polynomial time. The algorithms are exact and
	use rational linear programming, so no numerical tolerance is required.
	
	(2) We address Question~\ref{que:4.10} by giving a structural characterization
	of the vertices of $\Theta(G)$. Theorem \ref{thm:theta-vertices} shows that every vertex
	is uniquely determined by a full-rank active system consisting of the normalization equation \(x(E) = n-1\), tight rank constraints indexed by a laminar family, tight nonnegativity constraints, and tight degree constraints. This description reduces the study of the vertices and their denominators to explicit linear systems. Corollary \ref{cor:kG} gives the least positive integer $k_G$ for which $k_G \Theta(G)$ is a lattice polytope.
	
	(3) We give an affirmative answer to Question~\ref{que:6.9}. More precisely, Theorem~\ref{thm:path-products-RN} shows that every finite Cartesian product $P_{n_1}\square\cdots\square P_{n_d}$ is $\mathrm{RN}$, while Corollary~\ref{cor:RP-SRN-products} determines exactly which such products are $\mathrm{RP}$ and which are $\mathrm{SRN}$.
    
	\subsection{Organization of the paper}
	The remainder of the paper is organized as follows. Section 2 proves the polynomial time recognition theorem. Section 3 studies the vertices of \(\Theta(G)\) and  its minimal lattice dilation. Section 4 establishes the results on Cartesian products of paths.
	
	\section{Polynomial time recognition of RN, RP, and SRN graphs}
	By (\ref{eq:RN-RP-polytope}), recognizing RN, RP and SRN graphs amounts to deciding whether the relative interior \(P(G)^\circ\) contains a point satisfying the corresponding degree constraints. Devriendt \cite[Remark 3.10]{DevriendtRN} says that RN can be tested by linear programming, but it mentions the usual numerical issue caused by the relative interior condition. Thus we eliminate this problem by introducing an explicit slack variable. Therefore, we prove that all three recognition problems are solvable in polynomial time by linear programming.
	
	\begin{theorem}\label{thm:complexity}
		Given a finite connected graph \(G\), each of the properties RN, RP, and SRN can be decided in polynomial time.
	\end{theorem}
	
	\begin{proof}
		By~(\ref{eq:RN-RP-polytope}), \(G\) is RN if and only if there exists a point  \(x \in P(G)^\circ\)  such that  \(d_v(x) \leq 2 \) for every \(v \in V\). We first construct a rational point  \(a \in P(G)^\circ \) in polynomial time. Let \(\tau(G)\) denote the number of spanning trees of \(G\), and \(\tau_e(G)=\left| \left\{ T \in \mathcal{T}(G) : e \in T \right\} \right|\) denote the number of spanning trees containing \(e\).
		Define \[ a_e=\frac{\tau_e(G)}{\tau(G)}\qquad \text{for every } e \in E.\]
		 Equivalently, \[ a = \frac{1}{\tau(G)} \sum_{T \in \mathcal{T}(G)} \chi^T,\] 	where \(\frac{1}{\tau(G)}>0\).
		 Thus \(a\) is a strict convex combination of all the vertices of the spanning tree polytope \(P(G)\), and hence \(a \in P(G)^\circ\).
		By the matrix tree theorem, $\tau(G)$ and, for each edge $e$, $\tau_e(G)=\tau(G/e)$ can be computed exactly in polynomial time. Since these integers have polynomially bounded binary encoding lengths, the vector $a$ is a rational vector of polynomial encoding length and can be computed in polynomial time.
		
		We now formulate the RN decision problem as a linear programming (LP) problem.  The variables are \(w\in\R^E\) and \(\lambda\in\R\).  Consider
		\begin{equation}\label{eq:RN-LP}
			\begin{array}{rl}
				\text{maximize}   & \lambda \\[0.1cm]
				\text{subject to} & 0\le \lambda\le 1,\\[0.1cm]
				& w(E)=(1-\lambda)(n-1),\\[0.1cm]
				& w(F)\le (1-\lambda)r_G(F)\quad \text{for every }F\subseteq E,\\[0.1cm]
				& w_e\ge 0\quad \text{for every }e\in E,\\[0.1cm]
				& d_v(w)+\lambda d_v(a)\le 2\quad \text{for every }v\in V.
			\end{array}
		\end{equation}
        The constraints characterizing $w\in(1-\lambda)P(G)$ are precisely the scaled constraints defining the spanning tree polytope $P(G)$.
		Although the rank description of \(P(G)\) contains exponentially many inequalities, \(P(G)\), being the base polytope of the graphic matroid, has a polynomial time separation oracle \cite{CunninghamMP,SchrijverCP}. Here, a separation oracle is an algorithm which, given a rational vector \(x\), either certifies that \(x\in P(G)\) or returns a valid inequality for \(P(G)\) that is violated by \(x\).
		
		The scaled polytope \( (1-\lambda)P(G)\) also admits a polynomial time separation oracle.
	   If \(0\leq\lambda<1\), we apply the separation oracle to determine whether \(w/(1-\lambda)\in P(G)\). First check the necessary constraints
		\(w(E)=(1-\lambda)(n-1)\) and \(w_{e}\geq0\) for every \(e\in E\).
		If either condition fails, the corresponding violated equality or nonnegativity constraint separates \((w,\lambda)\) from the feasible region.
		Assume now that both conditions hold. Since \(1-\lambda>0\),
		\[
		w\in(1-\lambda)P(G)
		\quad\Longleftrightarrow\quad
		\frac{w}{1-\lambda}\in P(G).
		\]
        If $w/(1-\lambda)\notin P(G)$, the separation algorithm returns a subset $F\subseteq E$ such that
\[
\frac{w(F)}{1-\lambda}>r_G(F).
\]
It follows that $w(F)>(1-\lambda)r_G(F)$, so the valid scaled rank inequality
\(
w(F)\le(1-\lambda)r_G(F)
\)
is violated by the current point.

		If $\lambda = 1$, then $w(E) = 0$ and $w \ge 0$ imply $w = 0$.
		Consequently, all the scaled rank inequalities
	\(w(F) \le (1-\lambda) r_G(F) = 0\) are automatically satisfied.
		
		The remaining degree constraints \(d_v(w)+\lambda d_v(a)\leq 2 \) for every \( v\in V\) can be checked directly in polynomial time. Hence, the entire feasible region of (\ref{eq:RN-LP}) has a polynomial time separation oracle. Moreover, it is bounded, since
		$0 \le \lambda \le 1$, $w \ge 0$, and \(w(E) = (1-\lambda)(n-1) \le n-1\). Therefore, by the ellipsoid method and the optimization-separation principle, the feasibility and optimal value of (\ref{eq:RN-LP}) can be determined in polynomial time \cite{KhachiyanPA,GrotscheEM}.
		
	Next, we claim that $G$ is RN if and only if \eqref{eq:RN-LP} is feasible and its optimal value is positive.
		
		Suppose first that \eqref{eq:RN-LP} has a feasible solution with \(\lambda>0\).  If \(0<\lambda<1\), write \(z=w/(1-\lambda)\in P(G)\). Since \(w\in (1-\lambda)P(G)\), we have \(z\in P(G)\).
	    By the convexity property of the relative interior \cite[Theorem 6.1]{RockafellarCA}, it follows that:
		\[
		y=w+\lambda a=(1-\lambda)z+\lambda a \in P(G)^\circ.
		\]
		If \(\lambda=1\), the equality constraints imply \(w=0\), and hence \[
		y=w+\lambda a= a \in P(G)^\circ.
		\]
		Thus, in both cases, \(y\in P(G)^\circ\).
		
		According to the last degree constraint in equation \eqref{eq:RN-LP}:
		\[
		d_v(y)=d_v(w)+\lambda d_v(a)\le 2
		\quad\text{for every }v\in V.
		\]
		Consequently,
		\[y\in P(G)^\circ\cap \{x:d_v(x)\le 2\text{ for every }v\in V\}.\]
		By \eqref{eq:RN-RP-polytope}, \(G\) is RN.
		
		Conversely, suppose that \(G\) is RN.  Then there is a point \(y\in P(G)^\circ\) with \(d_v(y)\le 2\) for every \(v\).  Since $y,a \in  P(G)^\circ$, for sufficiently small \(\lambda>0\) the point
		\[
		z=\frac{y-\lambda a}{1-\lambda}
		=y+\frac{\lambda}{1-\lambda}(y-a)
		\in P(G).\]
		Let \(w=(1-\lambda)z=y-\lambda a\).  Then \(w\in (1-\lambda)P(G)\) and
		\(d_v(w)+\lambda d_v(a)=d_v(y)\le 2\).
		Thus \((w,\lambda)\) is feasible for \eqref{eq:RN-LP} with \(\lambda>0\).
		
		This proves the claim and shows that the decision problem for RN is solvable in polynomial time.
		
		We next consider RP and solve the following LP:
		\begin{equation}\label{eq:RP-LP}
			\begin{array}{rl}
				\text{maximize}   & \delta \\[0.1cm]
				\text{subject to} & x\in P(G),\\[0.1cm]
				& d_v(x)\le 2-\delta\quad \text{for every }v\in V.
			\end{array}
		\end{equation}
		This program is also solvable in polynomial time using the same separation oracle for \(P(G)\). Let \(\delta^*\) denote its optimal value. We claim that $G$ is RP if and only if $\delta^*>0$.
		
		If \(G\) is RP, then there exists a point \(x\in P(G)^\circ\) satisfying \(d_v(x)<2\) for every $v\in V$. At this point, set $\delta_0 = \min_{v \in V} \{2 - d_v(x)\}$. Then $(x, \delta_0)$ is feasible for (\ref{eq:RP-LP}), and hence $\delta^* \ge \delta_0 > 0$.
		
		Conversely, suppose that \(\delta^*>0\), and let \((x,\delta^\ast)\) be an optimal solution of (\ref{eq:RP-LP}). Then \(d_v(x) \le 2 - \delta^*\), \(v \in V\). We may choose \(0<\varepsilon<1\) sufficiently small that
		\[
		\varepsilon \max_{v \in V} |d_v(a) - d_v(x)| < \delta^*.
		\]
		Since $x \in P(G)$ and $a \in P(G)^\circ$, we have $y = (1-\varepsilon)x + \varepsilon a \in P(G)^\circ$. For every $v \in V$,
		\begin{align*}
			d_v(y)
			&= d_v(x) + \varepsilon (d_v(a) - d_v(x)) \\
			&\leq d_v(x) + \varepsilon|d_v(a) - d_v(x)| \\
			&\leq d_v(x) + \varepsilon \max_{v \in V} |d_v(a) - d_v(x)| \\
			&< 2 - \delta^* + \delta^* \\
			&= 2.
		\end{align*}
			Thus, $y \in P(G)^\circ$ and $d_v(y) < 2$ for every $v \in V$. By \eqref{eq:RN-RP-polytope}, $G$ is RP. Therefore, $G$ is RP if and only if $\delta^* > 0$, and RP graphs can be recognized in polynomial time.
		
		Finally, by definition, $G$ is $\mathrm{SRN}$ if and only if $G$ is $\mathrm{RN}$ and $G$ is not $\mathrm{RP}$. The polynomial time tests for RN and RP therefore yield a polynomial time test for SRN.
	\end{proof}
	
	\section{Vertices and Minimal Dilation Factors of \(\Theta(G)\)}
	
	This section gives a finite characterization of the vertices of \(\Theta(G)\).
	The key point is that all tight graphic matroid rank rows at a point are spanned by a chain of tight sets. We then combine this chain representation with the standard active constraint for vertices. Throughout this section, let \(G = (V, E)\) be a finite connected graph with \(n =|V|\).
	
	\subsection{Active systems}
	
	For \(F\subseteq E\), let \(\chi^F\in\{0,1\}^E\) be its incidence vector.  For \(v\in V\), let \(\chi^{E_{v}}\in\{0,1\}^E\) be the incidence vector of the set \(E_{v}\) of edges incident to \(v\).  If \(Z\subseteq E\) and \(f\in Z\), write \(e^f\in\{0,1\}^Z\) for the standard unit row vector corresponding to $f$, that is,
	\[
	(e^{f})_g =
	\begin{cases}
		1, & g = f, \\
		0, & g \ne f,
	\end{cases}
	\qquad \text{for every } g \in Z.
	\]
	
	A family \(\cL\subseteq 2^E\) is \emph{laminar} if for every \(A,B\in\cL\), the two sets are either disjoint or one contains the other.
	
	\begin{lemma}\label{lem:uncross}
		Let \(x\in P(G)\), and let
		\[
		\mathcal{T}_x=\{F\subseteq E:x(F)=r_G(F)\}
		\]
		be the family of edge sets whose rank inequalities are tight at \(x\).
		There exist nested tight sets \(\emptyset=C_0\subsetneq C_1\subsetneq \cdots\subsetneq C_q=E\), such that, for every $F \in \mathcal{T}_x$, there are real numbers
		$\alpha_1, \ldots, \alpha_q$ satisfying
		\[
		\chi^F = \sum_{i=1}^q \alpha_i \chi^{C_i}.
		\]
		Consequently,
		the incidence vectors $\chi^{C_1},\ldots,\chi^{C_q}$ span the incidence vectors of all rank constraints that are tight at \(x\). In particular, this vector space is spanned by the incidence vectors
		of a laminar subfamily of $\mathcal{T}_x$.
	\end{lemma}
	
	\begin{proof}
		First note that \(\mathcal{T}_x\) is closed under union and intersection.  Let \(A,B\in\mathcal{T}_x\), then
		\[x(A)+x(B)=r_G(A)+r_G(B) \qquad
		\text{and}  \qquad
	x(A)+x(B)=x(A\cap B)+x(A\cup B).\]
		Since \(x\in P(G)\),
		\(x(A\cap B)\le r_G(A\cap B)\) and \(x(A\cup B)\le r_G(A\cup B)\).
		We therefore obtain:
		\begin{align*}
			r_G(A) + r_G(B) = x(A \cap B) + x(A \cup B) \le r_G(A \cap B) + r_G(A \cup B).
		\end{align*}
		By submodularity of the matroid rank function, \cite[Theorem 39.8]{SchrijverCP} gives:
		\[
		r_G(A\cap B)+r_G(A\cup B)\le r_G(A)+r_G(B).
		\]
	 Consequently, all of the above inequalities are equalities. It follows that
		\(A\cap B \in\mathcal{T}_x,\; A\cup B\in\mathcal{T}_x\).
		Since \(\mathcal{T}_x\) is a finite family of subsets closed under union and intersection, $(\mathcal{T}_x,\cap,\cup)$ is a finite distributive lattice \cite{GratzerDL}.
		
		 Choose a maximal chain in \(\mathcal{T}_x\), ordered by inclusion:
		\(\emptyset=C_0\subsetneq C_1\subsetneq \cdots\subsetneq C_q=E\).
		 Define \( B_i=C_i\setminus C_{i-1}\) for each \(i \in \{1, \dots , q\}\).
		Since the sets $C_0, C_1, \ldots, C_q$ are nested, the sets
		$B_1, \ldots, B_q$ are pairwise disjoint. Moreover, because
		$C_0 = \emptyset$ and $C_q = E$, we have
		\[
		\bigcup_{i=1}^q B_i = \bigcup_{i=1}^q (C_i \setminus C_{i-1})
		= C_q \setminus C_0 = E.
		\]
		Thus, the collection $\{B_1, \ldots, B_q\}$ forms a partition of $E$.
	
	We next show that, for every $F \in \mathcal{T}_x$ and every
	$i \in \{1, \ldots, q\}$, either $B_i \subseteq F$ or $B_i \cap F = \emptyset$.
	
	 Suppose, to the contrary, that for some \(i\), the set \(F\) contains some but not all elements of \(B_i\).  Define \(D=C_{i-1}\cup (F\cap C_i)\).
	since the lattice \(\mathcal{T}_x\) is closed under intersection and union, and \(F, C_{i}, C_{i-1}\in \mathcal{T}_x\), it follows that \(D \in \mathcal{T}_x\).
	Since \(F\cap B_i\neq\emptyset\) and \(B_i\nsubseteq F\), the set \(F\cap B_i\) is a nonempty proper subset of \(B_i\). Therefore, using \(D=C_{i-1}\cup(F\cap C_i)=C_{i-1}\cup(F\cap B_i)\), we obtain
	\(C_{i-1}\subsetneq D\subsetneq C_i\),
	which contradicts the maximality of the chain. Therefore, for each $i$, either $B_i \subseteq F$ or $B_i \cap F = \emptyset$. In other words, every \(F\in\mathcal{T}_x\) is a union of a subcollection of the \(B_i\), that is, there exists a set \(I\subseteq\{1,\ldots,q\}\) such that
    \(F = \bigcup_{i \in I} B_i\).
	
     As the blocks \(B_{i}\) are pairwise disjoint,  each \(\chi^F\) is a sum of some block vectors \(\chi^{B_i}\):
     \(\chi^F = \sum_{i \in I} \chi^{B_i}\).
     Moreover, since $B_i = C_i \setminus C_{i-1}$ and $C_{i-1} \subsetneq C_i$,
     \(\chi^{B_i} = \chi^{C_i} - \chi^{C_{i-1}}\),
     and hence,
     \(\chi^F = \sum_{i \in I} \left( \chi^{C_i} - \chi^{C_{i-1}} \right)\).
    Therefore, each \(\chi^F\) lies in the span of the chain vectors \(\chi^{C_1},\ldots,\chi^{C_q}\).

    Finally, the vectors \(\chi^{C_1},\ldots,\chi^{C_q}\) form a chain and the chain is laminar. Hence the vector space spanned by \(\{\chi^F:F\in\mathcal T_x\}\) is spanned by the incidence vectors of a laminar subfamily of \(\mathcal T_x\).

	\end{proof}
	
	\paragraph{Notation.} For $r_G(F) = n - \kappa_G(F)$, since $G$ is connected, $\kappa_G(E) = 1$, and hence $r_G(E) = n - 1$. Therefore, every point $x \in P(G)$ satisfies the global equality $x(E) = r_G(E) = n - 1$.
	
	In the chain
	\(\emptyset = C_0 \subsetneq C_1 \subsetneq \cdots \subsetneq C_q = E\),
	the final set $C_q = E$ corresponds to this global equality and will be treated separately. The initial set $C_0 = \emptyset$ corresponds to the trivial equality $x(\emptyset) = r_G(\emptyset) = 0$. Thus, $C_1, \ldots, C_{q-1}$ correspond to the proper nonempty tight rank constraints $x(C_i) = r_G(C_i)$ for $i = 1, \ldots, q - 1$.
	
	By the standard active constraint characterization of vertices of
	polyhedra (see \cite[Theorems 2.2 and 2.3]{BertTsitLP}), a point
	\(x\) in a nonempty polyhedron in \(\mathbb{R}^E\) is a vertex
	if and only if the normal vectors of the tight constraints at \(x\) span \(\mathbb{R}^E\). Equivalently, the tight constraints at
	\(x\) contain a subsystem whose coefficient matrix has rank \(|E|\);
	the corresponding system of equalities then has \(x\) as its unique
	solution. Therefore, \(x\in\Theta(G)\) is a vertex if and only if its
	tight constraints contain such a full-rank subsystem.
	
	\begin{theorem}\label{thm:theta-vertices}
		Let \(G=(V,E)\) be a connected graph.  A vector \(x\in\R^E\) is a vertex of \(\Theta(G)\) if and only if the following conditions hold:
		
		\begin{enumerate}[label=(\roman*)]
			\item  \(x\in P(G)\) and \(d_v(x)\le 2\) for every \(v\in V\).
			\item There exist a laminar family $\cL\subseteq 2^E\setminus\{\emptyset, E\}$, $Z\subseteq E$, and $ D\subseteq V$ such that \(x\) is the unique solution of the system:
			\begin{equation}\label{eq:active-system}
				\begin{aligned}
					x(E)&=n-1,\\
					x(F)&=r_G(F)       &&(F\in\cL),\\
					x_f&=0             &&(f\in Z),\\
					d_v(x)&=2          &&(v\in D).
				\end{aligned}
			\end{equation}
		\end{enumerate}
	\end{theorem}
	
	\begin{proof}
		 From the definition \eqref{eq:TMD-polytope}, condition \((i)\) clearly holds. We use the standard fact that a point of a polytope in \(\R^E\) is a vertex if and only if the normals of the active constraints at that point span \(\R^E\).
		
		Suppose first that \(x\) is a vertex of \(\Theta(G)\).  Then the normal vectors of the constraints active at \(x\) span \(\mathbb R^E\). Equivalently, the active constraints at \(x\) contain \(|E|\) linearly independent coefficient rows.
		Since \(G\) is connected, every point \(y\in P(G)\) satisfies \(y(E) = n-1\).
		Let
		\[
		\mathcal{T}_x=\{F\subseteq E:x(F)=r_G(F)\}.
		\]
		By Lemma \ref{lem:uncross}, there exists a chain  \(\emptyset=C_0\subsetneq C_1\subsetneq \cdots\subsetneq C_q=E\) of \(\mathcal{T}_x\) such that the vectors \(\chi^{C_1},\ldots,\chi^{C_q}\) span the incidence vectors of all tight rank constraints at \(x\).
		
		Let
	\(\mathcal L=\{C_1,\ldots,C_{q-1}\}\),
		which is a laminar subfamily of \(2^E\setminus\{\emptyset,E\}\).
		The final member \(C_{q} = E\) is represented separately by the global equality \(x(E) = n-1\).
		
		Define
		\[Z= \{f\in E: x_{f} =0\}\qquad \text{and} \qquad D = \{v\in V: d_{v}(x)=2\},\]
		where \(Z\) records all nonnegativity constraints active at \(x\), while \(D\) records all degree constraints active at \(x\).
		The coefficient vectors corresponding to \(\mathcal L\), \(Z\), and \(D\), together with \(\chi^{E}\), span all active constraint normals at \(x\).
		Since the active constraint normals at \(x\) span \(\mathbb{R}^E\),
		these coefficient vectors also span \(\mathbb{R}^E\).
		 Therefore, the coefficient matrix of (\ref{eq:active-system}) has rank \(|E|\). Moreover, since $x$ satisfies \eqref{eq:active-system}, it is the unique solution of the system.
	
		Conversely, suppose that \(x\) satisfies \((i)\) and \((ii)\). The \((i)\) implies that \(x \in \Theta(G)\). Moreover, every equation in \eqref{eq:active-system} corresponds to a tight constraint of \(\Theta(G)\) that is active at \(x\): the defining
		equality of \(P(G)\), the tight rank constraints, the tight
		nonnegativity constraints indexed by \(Z\), and the tight degree
		constraints indexed by \(D\).
		
		
		Suppose that
		\(x=\lambda y+(1-\lambda)z \)
		for some \(y,z\in\Theta(G)\) and \(0<\lambda<1\). Since all equations
		in \eqref{eq:active-system} correspond to constraints that are tight
		at \(x\), both \(y\) and \(z\) also satisfy
		\eqref{eq:active-system}. By the uniqueness of the solution to
		\eqref{eq:active-system}, we obtain \(y=z=x\). Therefore, \(x\) is a vertex of \(\Theta(G)\).
	\end{proof}
	
	\subsection{Minimal Dilation Factors of \(\Theta(G)\) to Lattice Polytope}
	
	For a rational vector \(x\), let \(\den(x)\) denote the least positive integer \(p\) such that \(px\in\Z^E\).  A polytope is called a \emph{lattice polytope} if its vertices are integer points.
	Suppose that $\Theta(G) \neq \emptyset$. We define minimal dilation factors
	\[
	k_G = \operatorname{lcm}\{\operatorname{den}(x) : x \in V(\Theta(G))\},
	\] where \(\operatorname{lcm}\) denotes the least common multiple.
	Equivalently, $k_G$ is the least positive integer $k$ such that $k\Theta(G)$ is a lattice polytope, where $\Theta(G) \neq \emptyset$.
\begin{corollary}\label{cor:kG}
Let $G = (V, E)$ be connected with $|E|=m$.
Let \(\mathcal C\) denote the collection of any set of \(m\) equations, with linearly independent coefficient vectors, selected from a system of the form \eqref{eq:active-system}. Then
\begin{equation}\label{eq:kG}
k_G = \min\{k \in \mathbb N : k\Theta(G) \text{ is a lattice polytope}\}
= \operatorname{lcm}\{\operatorname{den}(x_{\mathcal C}) : x_{\mathcal C} \in \Theta(G)\}.
\end{equation}
\end{corollary}

	\begin{proof}
		For each such collection $\mathcal C$, write the corresponding system as $A_{\mathcal C} x = b_{\mathcal C}$. Since \(A_{\mathcal C}\) is an \(m\times m\) matrix of rank \(m\), it is nonsingular. Hence, the subsystem has the unique solution
		\[
		x_{\mathcal C}:=A_{\mathcal C}^{-1}b_{\mathcal C}.
		\] By Theorem \ref{thm:theta-vertices}, the feasible vectors \(x_{\mathcal C}\) obtained from the nonsingular systems $A_{\mathcal C} x = b_{\mathcal C}$
		are precisely the vertices of \(\Theta(G)\), possibly with repetitions. Therefore, \(\operatorname{lcm}\{\operatorname{den}(x_{\mathcal C}) : x_{\mathcal C} \in \Theta(G)\}\) is equal to \(\operatorname{lcm}\{\operatorname{den}(x) : x \in V(\Theta(G))\}\). By the definition of \(k_{G}\), this is exactly the least positive
		integer \(k\) such that \(k\Theta(G)\) is a lattice polytope.
	\end{proof}
	
	\begin{example}\label{ex:denominator-three}
		 The following example shows that a vertex of the TDM polytope may have
		 denominator three. Let \(G\)  be the graph with vertex set \(\{0,1,\ldots,8\}\) and edge set $E=\{e_{1},\dots,e_{14}\}$, where the edges are ordered
		 as follows:
		\[
		\begin{array}{c|cccccccccccccc}
			i&1&2&3&4&5&6&7&8&9&10&11&12&13&14\\ \midrule
			e_i&01&08&12&14&23&34&36&37&45&46&47&56&67&78
		\end{array}
		\]
		Consider a point with denominator three:
		\[
		x=\left(
		0,1,\frac{2}{3},1,1,\frac{1}{3},\frac{1}{3},\frac{1}{3},\frac{2}{3},0,0,1,\frac{2}{3},1
		\right).
		\]
		First, we show that \(x\in\Theta(G)\). Let $T_1,T_2,T_3$ be the subgraphs of $G$ with edge sets
		\begin{align*}
			E(T_1) &= \{e_2, e_3, e_4, e_5, e_7, e_8, e_{12}, e_{14}\}, \\
			E(T_2) &= \{e_2, e_3, e_4, e_5, e_9, e_{12}, e_{13}, e_{14}\}, \\
			E(T_3) &= \{e_2, e_4, e_5, e_6, e_9, e_{12}, e_{13}, e_{14}\}.
		\end{align*}
		The graphs \(T_1, T_2, T_3\) are shown in
		Figure~\ref{fig:example}. It is clear from the figure that
		\(T_1, T_2, T_3\) are spanning trees of \(G\).
		\begin{figure}[htbp]
			\centering
			\includegraphics[width=0.8\textwidth,
			trim={0 3cm 0 3cm},
			clip]{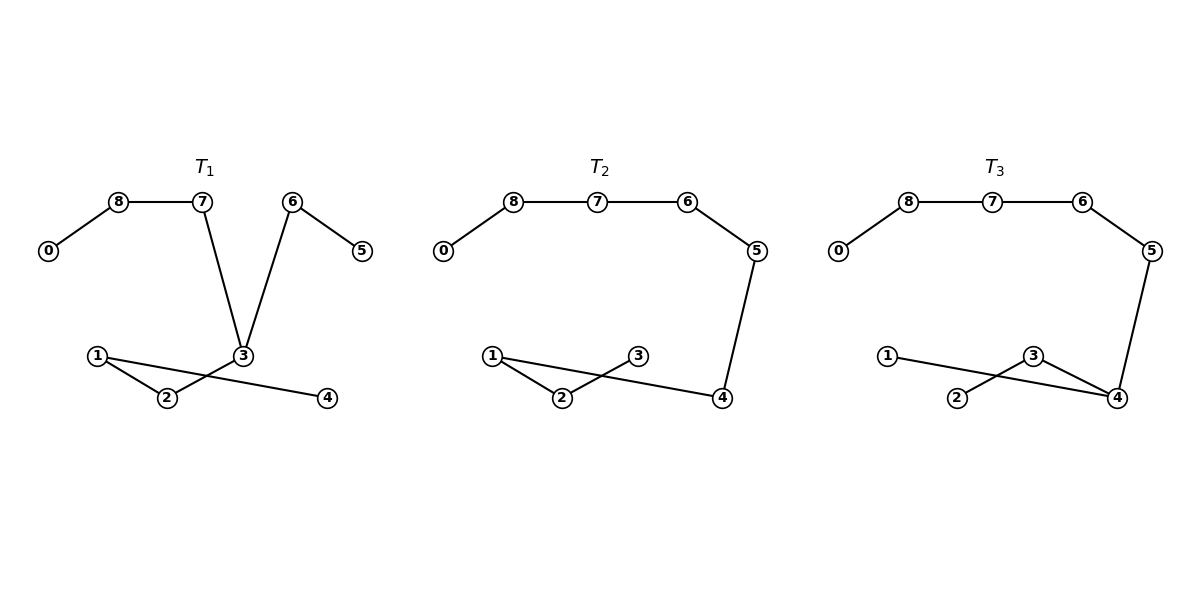}
			\caption{The spanning trees \(T_{1}\), \(T_{2}\), \(T_{3}\).}
			\label{fig:example}
		\end{figure}
	
	Moreover, the incidence vectors of the three spanning trees are:
		\begin{align*}
			\chi^{T_1} &= \{0,1,1,1,1,0,1,1,0,0,0,1,0,1\}, \\
			\chi^{T_2} &= \{0,1,1,1,1,0,0,0,1,0,0,1,1,1\}, \\
			\chi^{T_3} &= \{0,1,0,1,1,1,0,0,1,0,0,1,1,1\}.
		\end{align*}
		Therefore \(x=\frac{1}{3}(\chi^{T_1}+\chi^{T_2}+\chi^{T_3})\), and hence $x\in P(G)$.
		
		The degree vector of \(x\) is
		\((1,\frac{5}{3},\frac{5}{3},2,2,\frac{5}{3},2,2,2)\),
		so \(d_{v}(x)\leq 2\) for every \(v\in V(G)\). Therefore $x\in\Theta(G)$.
		
		We next prove that $x$ is a vertex of \(\Theta(G)\). For $S\subseteq V$, let \(E(S)\) denote the set of edges of \(G\) with both endvertices in \(S\). Consider the following fourteen constraints that are active at \(x\):
		\begin{table}[htbp]
			\centering
			\begin{tabular}{l l}
				\toprule
				Type & Active equations (see (\ref{eq:active-system})) \\
				\midrule
				\(x(E)=n-1\) & \( x(E) = 8 \) \\
				\midrule
				\(x(E(S))=|S|-1\) & \( x(E(\{0, 8\})) = 1, \quad x(E(\{1, 4\})) = 1, \) \\
				& \( x(E(\{2, 3\})) = 1, \quad x(E(\{5, 6\})) = 1, \) \\
				& \( x(E(\{1, 2, 3, 4\})) = 3 \) \\
				\midrule
				\(x_f=0\) & \( x_{01} = x_{46} = x_{47} = 0 \) \\
				\midrule
				\(d_v(x)=2\) & \( d_3(x) = d_4(x) = d_6(x) = d_7(x) = d_8(x) = 2 \) \\
				\bottomrule
			\end{tabular}
		\end{table}
		
		Let \(A\) be the coefficient matrix of these fourteen active constraints, where the rows are ordered as
		displayed above and the columns are ordered according to $e_1,\ldots,e_{14}$:
		{\scriptsize
			\[
			\setlength{\arraycolsep}{2pt}
			A=\left(\begin{array}{rrrrrrrrrrrrrr}
				1&1&1&1&1&1&1&1&1&1&1&1&1&1\\
				0&1&0&0&0&0&0&0&0&0&0&0&0&0\\
				0&0&0&1&0&0&0&0&0&0&0&0&0&0\\
				0&0&0&0&1&0&0&0&0&0&0&0&0&0\\
				0&0&0&0&0&0&0&0&0&0&0&1&0&0\\
				0&0&1&1&1&1&0&0&0&0&0&0&0&0\\
				1&0&0&0&0&0&0&0&0&0&0&0&0&0\\
				0&0&0&0&0&0&0&0&0&1&0&0&0&0\\
				0&0&0&0&0&0&0&0&0&0&1&0&0&0\\
				0&0&0&0&1&1&1&1&0&0&0&0&0&0\\
				0&0&0&1&0&1&0&0&1&1&1&0&0&0\\
				0&0&0&0&0&0&1&0&0&1&0&1&1&0\\
				0&0&0&0&0&0&0&1&0&0&1&0&1&1\\
				0&1&0&0&0&0&0&0&0&0&0&0&0&1
			\end{array}\right),
			\]
		}
		
		Its determinant is $-3$. Hence, \(A\) is nonsingular, and the coefficient vectors of the fourteen active constraints are linearly independent. Since \(|E|=14\), these vectors span \(\mathbb R^E\). Therefore, by Theorem \ref{thm:theta-vertices}, \(x\) is a vertex of \(\Theta(G)\).
		
		 Since \(x\) has denominator three, the dilation factor \(k\) such that
		 \(k\Theta(G)\) is a lattice polytope must be divisible by three.
	\end{example}
	
	\section{The Cartesian Product of Paths}

In this section, $G\square G'$ denotes the Cartesian product of two graphs $G$ and $G'$. Let $\mathbf n=(n_1,\ldots,n_d)$, where $d\ge2$ and $n_i\ge2$ for every $i\in\{1,\ldots,d\}$, and define $B(\mathbf n)=P_{n_1}\square\cdots\square P_{n_d}$. Each path $P_{n_i}$ is called a factor of $B(\mathbf n)$. The vertex set of $B(\mathbf n)$ is
\[
V(B(\mathbf n))=\{(a_1,\ldots,a_d):1\le a_i\le n_i\text{ for every }i\in\{1,\ldots,d\}\},
\]
and two vertices are adjacent if they differ by one in exactly one coordinate and agree in all other coordinates. In particular, when $d=2$, the graph $B(\mathbf n)$ is called a grid graph. For related studies of resistance curvature on Cartesian products of graphs, see \cite{Daw,DevriendtRN}.
	
    \subsection{Hamiltonian paths and common intervals}
	
	For a coordinate \(i\), an \emph{\(i\)-row} is a set of all vertices obtained by fixing all coordinates except the \(i\)-th coordinate and allowing only the \(i\)-th coordinate to vary. That is, an \(i\)-row is a set of the form
	\(
\{(a_1,\ldots,a_{i-1},k,a_{i+1},\ldots,a_d):1\le k\le n_i\},
\)
where $a_j\in\{1,\ldots,n_j\}$ is fixed for every $j\neq i$.
	
	A subset \(S\) of the vertices of a path \(H\) is called an \emph{interval} of \(H\) if the vertices of \(S\) occur consecutively along \(H\).  Equivalently, the subgraph of \(H\) induced by \(S\) has exactly \(|S|-1\) edges. We say that \(S\) is \emph{saturated} in the \(i\)-direction if every \(i\)-row meeting \(S\) is entirely contained in \(S\).
	
	\begin{lemma}\label{lem:H-path}
		\cite{LiYeZhangIS} For any positive integers \(d\) and \(n_{1}, \dots , n_{d}\), the Cartesian product \(P_{n_1}\square\cdots\square P_{n_d}\) has a Hamiltonian path.
	\end{lemma}
	
		\begin{lemma}\label{lem:opposite-snakes}
		For every coordinate \(i\in \{1,\dots, d\}\), there are two Hamiltonian paths \(\overrightarrow{P}_{i}\) and \(\overleftarrow{P}_{i}\) in \(B(\mathbf n)\) and an ordering $
		L_1,L_2,\ldots,L_q$ of the \(i\)-rows, where $ q=\prod_{j\ne i}n_j$, such that the following properties hold.
		\begin{enumerate}[label=(\roman*)]
			\item For every \(t \in \{1, . . . , q\}\), all vertices of \(L_{t}\) occur consecutively along each of \(\overrightarrow{P}_{i}\) and \(\overleftarrow{P}_{i}\).
			\item The two paths \(\overrightarrow{P}_{i}\) and \(\overleftarrow{P}_{i}\) traverse the \(i\)-row in the same order \(L_{1}, \dots , L_{q}\).
			\item For each \(t\in\{1,\ldots,q\}\), \(\overrightarrow{P}_{i}\) and \(\overleftarrow{P}_{i}\) traverse the vertices of \(L_t\) in opposite orders.
		\end{enumerate}
	\end{lemma}
	
	\begin{proof}
		Contract each \(i\)-row to a vertex.  The resulting quotient graph is isomorphic to
		$G_{i}=\prod_{j\ne i} P_{n_j}$, which has $ q=\prod_{j\ne i}n_j$ vertices and admits a Hamiltonian path by the usual recursive snake construction. Each vertices of the \(G_{i}\) corresponds to an $L_t$ of $i$-row.
		
		Let \(v_1,\ldots,v_q\) be the vertices of a Hamiltonian path in \(G_i\), where
		\(v_t=
		\bigl(a_1^{(t)},\ldots,a_{i-1}^{(t)},
		a_{i+1}^{(t)},\ldots,a_d^{(t)}\bigr)\),
		and let \(L_1,\ldots,L_q\) be the corresponding \(i\)-rows, where
		\(L_t = \{\ell_t(s) : 1 \leq s \leq n_i\}\) and define \(
		\ell_t(s):=
		\bigl(a_1^{(t)},\ldots,a_{i-1}^{(t)},s,
		a_{i+1}^{(t)},\ldots,a_d^{(t)}\bigr)\) with \(1\le s\le n_i\).
		If \(v_t\) and \(v_{t+1}\) are adjacent in \(G_i\), then they differ
		by one in exactly one coordinate \(j\neq i\) and agree in all other
		coordinates. For every \(s\in\{1,\ldots,n_i\}\), the
		vertices \(\ell_t(s)\) and \(\ell_{t+1}(s)\) differ by one only in
		their \(j\)-th coordinate, hence are adjacent in \(B(\mathbf n)\).
		In particular, the two consecutive \(i\)-rows \(L_t\) and
		\(L_{t+1}\) can be connected through the pairs
		\(\ell_t(1),\ell_{t+1}(1)\) or \(\ell_t(n_i),\ell_{t+1}(n_i)\).

		Define \(\overrightarrow{P}_{i}\) by traversing \(L_t\) in the order
		\(\ell_t(1), \ell_t(2), \ldots, \ell_t(n_i)\)
		when \(t\) is odd, and in the reverse order
		\(\ell_t(n_i), \ell_t(n_i-1), \ldots, \ell_t(1)\) when \(t\) is
		even.
		Consecutive vertices
		within each \(L_{t}\) are adjacent, since their \(i\)-th coordinates differ by one and all other coordinates are equal. Moreover,
		if \(t\) is odd, the traversal of \(L_t\) ends at \(\ell_t(n_i)\) and
		that of \(L_{t+1}\) begins at \(\ell_{t+1}(n_i)\); if \(t\) is even,
		the corresponding vertices are \(\ell_t(1)\) and
		\(\ell_{t+1}(1)\). In either case, these two vertices \(\ell_t(n_i)\) and \(\ell_{t+1}(n_i)\) (or \(\ell_t(1)\) and \(\ell_{t+1}(1)\)) are adjacent.
		Hence the traversals of \(L_1,\ldots,L_q\) concatenate to form a path.
	   Since the \(i\)-rows \(L_1,\ldots,L_q\) are pairwise disjoint and
		partition \(V(B(\mathbf n))\), this path visits every vertex exactly
		once. Therefore, \(\overrightarrow{P}_{i}\) is a Hamiltonian path.
		
		We next construct \(\overleftarrow{P}_{i}\) using the same order \(L_1,\ldots,L_q\), but
		reverse the traversal direction on every \(i\)-row. Hence \(L_t\) is
		traversed from \(\ell_t(n_i)\) to \(\ell_t(1)\) when \(t\) is odd, and
		from \(\ell_t(1)\) to \(\ell_t(n_i)\) when \(t\) is even. By the same
		argument as above, these traversals concatenate to form a Hamiltonian
		path \(\overleftarrow{P}_{i}\).
		
		By construction, in each of the paths \(\overrightarrow{P}_{i}\) and \(\overleftarrow{P}_{i}\), the
		vertices of every \(L_t\) appear consecutively. Moreover,
		the two paths traverse the \(i\)-rows in the same order
		\(L_1,\ldots,L_q\), but traverse the vertices within each \(L_t\) in
		opposite orders. Therefore, \(\overrightarrow{P}_{i}\) and \(\overleftarrow{P}_{i}\) satisfy
		conditions \emph{(i)--(iii)}.
	\end{proof}
	
	\begin{example}
	The two Hamiltonian paths \(\overrightarrow{P}_{i}\) and \(\overleftarrow{P}_{i}\) in the grid graph \(P_3\square P_4\) are illustrated in Figure~\ref{fig:path}, where \(i=2\).
		\begin{figure}[htbp]
			\centering
			\begin{minipage}{0.48\textwidth}
				\centering
				\includegraphics[width=\textwidth]{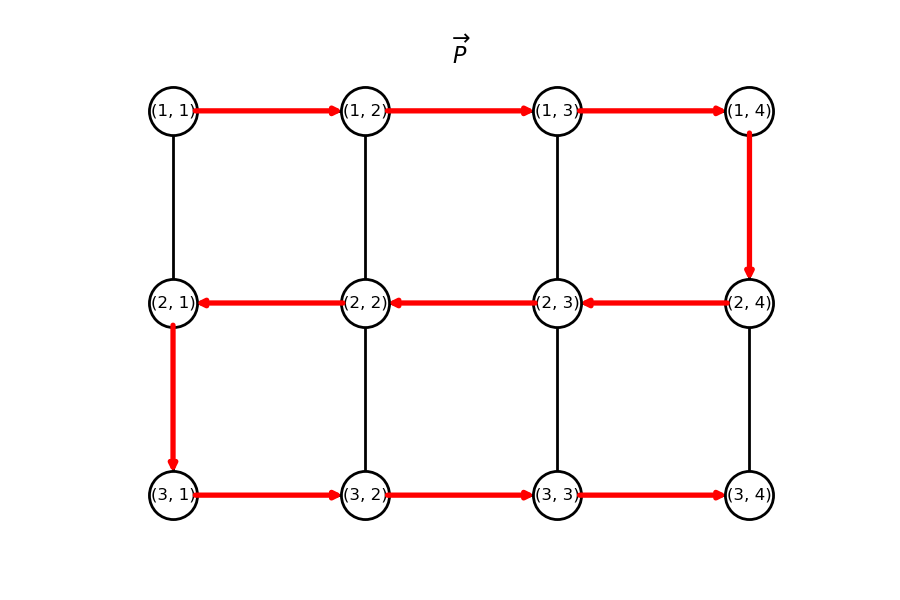}
			\end{minipage}
			\hfill
			\begin{minipage}{0.48\textwidth}
				\centering
				\includegraphics[width=\textwidth]{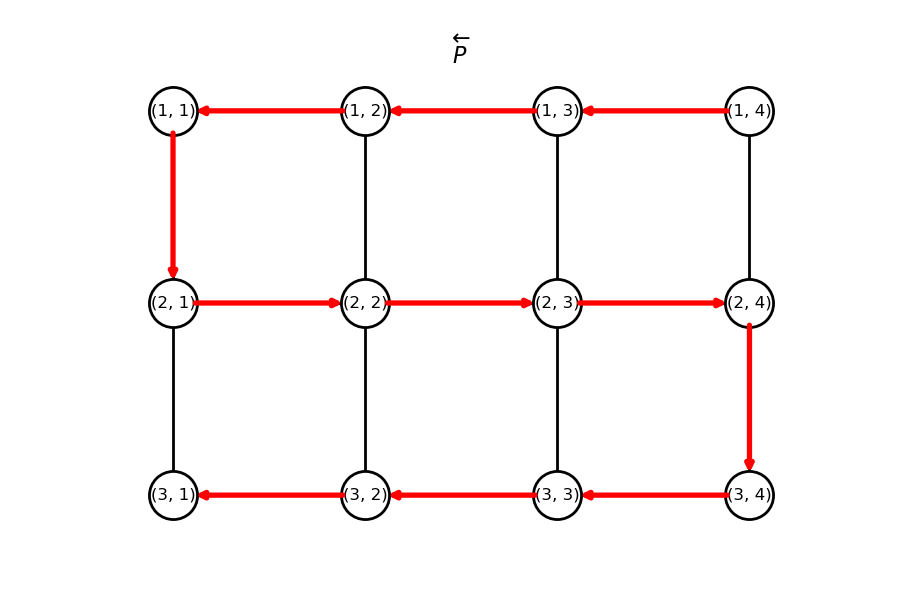}
			\end{minipage}
			\caption{Hamiltonian paths \(\overrightarrow{P}_{2}\) and \(\overleftarrow{P}_{2}\) in \(P_3\square P_4\).}
			\label{fig:path}
		\end{figure}
		
	In the grid graph shown in Figure~\ref{fig:path}, the vertex set of each row forms a
	\(2\)-row, whereas the vertex set of each column forms a \(1\)-row.
	For example,  all \(2\)-rows are
	\begin{align*}
		L_1=\{(1,1),(1,2),(1,3),(1,4)\}, \\
		L_2=\{(2,1),(2,2),(2,3),(2,4)\}, \\
		L_3=\{(3,1),(3,2),(3,3),(3,4)\},
	\end{align*}
	where in the notation introduced above,
	\(\ell_1(2)=(1,2)\).	
	\end{example}

	\begin{lemma}\label{lem:common-intervals}
		Let \(S\subseteq V(B(\mathbf n))\).  If \(S\) is an interval of every path \(\overrightarrow{P}_{i}\) and \(\overleftarrow{P}_{i}\), \(i=1,\ldots,d\), then \(S\) is empty, a singleton, or the whole vertex set.
	\end{lemma}
	
	\begin{proof}
		Fix a coordinate \(i\).  We first state and prove the following claim.
		
		\medskip
		\noindent\emph{Claim:} If \(S\) is an interval of both \(\overrightarrow{P}_{i}\) and \(\overleftarrow{P}_{i}\), and intersects at least two distinct \(i\)-rows, then every such \(i\)-row is entirely contained in \(S\).
		
		\medskip
		\noindent\emph{Proof of the claim.}
		In both \(\overrightarrow{P}_{i}\) and \(\overleftarrow{P}_{i}\), each \(i\)-row \(L_t\) forms a consecutive segment, and the segments occur in the same order	\(L_1,\ldots,L_q\).
		Since \(S\) is an interval in both paths, the \(i\)-rows intersecting \(S\) form a consecutive sequence
		\(L_t,L_{t+1},\ldots,L_{t+r}\).
		Therefore, every intermediate \(i\)-row \(L_{t+1},\ldots,L_{t+r-1}\) is completely contained in \(S\).
		
		Now consider the first \(L_t\) that intersects \(S\). Since \(S\) is an interval in \(\overrightarrow{P}_{i}\) and extends to a later \(i\)-row, the set \(S\cap L_t\) is a nonempty suffix of \(L_t\). Similarly, \(S\cap L_t\) is a nonempty suffix of \(L_t\) in \(\overleftarrow{P}_{i}\).
		Since \(\overleftarrow{P}_{i}\) traverses \(L_t\) in the opposite direction, \(S\cap L_t\) is also a nonempty prefix of \(L_t\). Thus, \(S\cap L_t = L_t\), so the whole \(L_t\) is contained in \(S\). The same argument applies to \(L_{t+r}\).
		Therefore, every \(i\)-row intersecting \(S\) is completely contained in \(S\). We also say that \(S\) is saturated in the \(i\)-direction, which proves the claim.
		
		The empty set and all singletons are clearly intervals of \(\overrightarrow{P}_{i}\) and \(\overleftarrow{P}_{i}\). Now suppose that \(S\) contains two distinct vertices \(u\) and \(v\).  We show that \(S\) must be the whole vertex set.
		
		Fix any coordinate \(k\).  If \(u\) and \(v\) differ in some coordinate other than \(k\), then they lie on two different \(k\)-rows, so \(S\) meets at least two \(k\)-rows.  According to the claim, \(S\) is saturated in the \(k\)-direction.
		
		If $u$ and $v$ differ only in coordinate \(k\),
		so they lie on the same \(k\)-row. Since \(d\ge 2\), choose
		\(i\neq k\), then \(u\) and \(v\) lie on distinct \(i\)-rows, so the claim implies that \(S\) is saturated in the \(i\)-direction. Since \(n_i\ge 2\),
		the \(i\)-row through \(u\) contains a vertex \(u'\in S\) with \(u'\neq u\). The vertices \(u\) and \(u'\) lie on distinct \(k\)-rows. Therefore, by the preceding claim, \(S\) is saturated in the \(k\)-direction.
		Thus, \(S\) is saturated in every coordinate direction.
	
	 Take any vertex \(x=(x_1,\ldots,x_d)\in S\).
	 For an arbitrary target vertex \(y=(y_1,\ldots,y_d)\in V(B(\mathbf n))\), starting from \(x\), we successively replace \(x_1\) by \(y_1\), then \(x_2\) by \(y_2\), and continue in this way until all coordinates have been replaced by the corresponding coordinates of \(y\).
	 At each step, only one coordinate is changed, and the saturation of \(S\) in every coordinate direction guarantees that the resulting vertex remains in \(S\). Finally, we obtain \(y\in S\). Since \(y\) was arbitrary,
	  \(S=V(B(\mathbf n))\). As shown above, every common interval \(S\) containing at least two vertices is equal to \(V(B(\mathbf n))\).
	\end{proof}
	
	\begin{lemma}\label{lem:interior-snakes}
		Let \(\mathcal{P}=\{\overrightarrow{P}_{i}, \overleftarrow{P}_{i}:i=1,\ldots,d\}\).
		The vector
		\[
		x=\frac{1}{2d}\sum_{H\in\mathcal{P}}\chi^H
		\]
		belongs to \(P(B(\mathbf n))^\circ\) and satisfies \(d_v(x)\le 2\) for every vertex \(v\).
	\end{lemma}
	
	\begin{proof}
	Every path \(H\in\mathcal P\) is a Hamiltonian path and hence a spanning
	tree of \(B(\mathbf n)\). Therefore,
	\(\chi^H\in P(B(\mathbf n))\) for every \(H\in\mathcal P\).
	Since \(|\mathcal P|=2d\) and \(x\) is the average of these incidence vectors, it follows that
	\(x\in P(B(\mathbf n))\) and \(x(E)= |V(B(\mathbf n))|-1\).

	Moreover, every vertex has degree at most two in a path, so
	\(d_v(\chi^H)\le 2\) for every \(H\in\mathcal P\) and \(v\in V(B(\mathbf n))\).
	Averaging over \(H\in\mathcal P\), we obtain
	\(d_v(x)\le 2\) for every \(v\in V(B(\mathbf n))\).
		
		Next, we shall prove that \(x\in P(B(\mathbf n))^\circ\).  Every edge \(e\in E(B(\mathbf n))\) belongs to an \(i\)-row for some \(i\in\{1,\ldots,d\}\). Since both \(\overrightarrow{P}_{i}\) and \(\overleftarrow{P}_{i}\) traverse every \(i\)-row completely, \(e\) belongs to both \(\overrightarrow{P}_{i}\) and \(\overleftarrow{P}_{i}\). Therefore,
		\(x_e>0\) for every \(e\in E(B(\mathbf n))\).
		
		It remains only to verify that \(x(E(S))<|S|-1\) for every \(S\subsetneq V\) with \(|S|\ge2\).
		
		Fix such an \(S\).  For a Hamiltonian path \(H\), the induced subgraph \(H[S]\) is a disjoint union of paths, so
		\(|E(H[S])|\leq |S|-1\),
		with equality if and only if $H[S]$ is connected, equivalently, if and only if \(S\) is an interval of \(H\).
		Suppose the average \(x(E(S))=|S|-1\). Then every path \(H\in\mathcal{P}\) would have exactly \(|S|-1\) edges inside \(S\), and hence \(S\) would be an interval of every \(\overrightarrow{P}_{i}\) and \(\overleftarrow{P}_{i}\).  By Lemma \ref{lem:common-intervals}, \(S\) would be a singleton or the whole vertex set, which contradicts \(|S|\geq 2\) and \(S\subsetneq V(B(\mathbf n))\).
	     Therefore, \(x(E(S))<|S|-1\). 	Consequently,
		\(x\in P(B(\mathbf n))^\circ\).
	\end{proof}
	
	The following two results immediately follow from
	Lemma~\ref{lem:interior-snakes}.
	\begin{theorem}\label{thm:path-products-RN}
		For \(d\ge2\) and \(n_i\ge2\) with \(i=1,\ldots,d\), the Cartesian product of the paths
		$P_{n_1}\square\cdots\square P_{n_d}$
		is RN.
	\end{theorem}
	
	\begin{corollary}\label{cor:grid-RN}
		\(P_m\square P_n\) is RN for all $m, n \in \mathbb{N}$.
	\end{corollary}

    \begin{corollary}\label{cor:RP-SRN-products} Let \(G=P_{n_1}\square\cdots\square P_{n_d}\), where \(d\ge2\) and \(n_i\ge2\) for \(i=1,\ldots,d\). Then the following statements hold: \begin{enumerate}[label=(\roman*)] \item If \(\prod_i n_i\) is even, then \(G\) is RP. \item If \(\prod_i n_i\) is odd, then \(G\) is SRN. \end{enumerate} \end{corollary}
\begin{proof}
	We first prove~{\rm (i)}. Assume that $\prod_i n_i$ is even. Then at least one $n_i$ is even. Choose $j\neq i$ and let $G_1=P_{n_i}\square P_{n_j}$. Since $n_in_j$ is even, the grid graph $G_1$ has a Hamiltonian cycle. Moreover, the Cartesian product of a Hamiltonian graph and a path is Hamiltonian. Applying this fact successively to the remaining path factors shows that $G$ is Hamiltonian. By~\cite[Theorem~6.2]{DevriendtRN}, every Hamiltonian graph is RP. Therefore, $G$ is RP.

We next prove~{\rm (ii)}. Assume that $\prod_i n_i$ is odd. It is clear that $G$ is bipartite. Since $G$ has an odd number of vertices, its two bipartition classes have different sizes. By~\cite[Proposition~3.5]{DevriendtRN}, a bipartite graph whose bipartition classes have different sizes is not RP. On the other hand, Theorem~\ref{thm:path-products-RN} shows that $G$ is RN. Therefore, $G$ is SRN.
\end{proof}	

\section{Conclusion}
In this paper, we resolved several problems posed by Devriendt concerning resistance nonnegative graphs. We established polynomial time recognition results for RN, RP, and SRN graphs, characterized the vertices of the tree double matching polytope and determined the minimal dilation factor that makes it a lattice polytope, and determined the resistance properties of finite Cartesian products of paths.

Several problems still remain open. For example, Devriendt asked which biconnected planar graphs are RN. By Corollary~\ref{cor:grid-RN}, every grid graph $P_m\square P_n$ is biconnected, planar, and RN. In contrast, it is not hard to verify that the complete bipartite graph \(K_{2,m}\) with integer \(m>3\) is biconnected and planar but is not RN. Therefore, not every biconnected planar graph is RN. Beyond this family, however, a broader structural characterization of biconnected planar RN graphs remains open. Moreover, the study of Cartesian products in this paper is restricted to products of paths. It would therefore be interesting to determine conditions under which Cartesian products of more general connected graphs are RN or RP.

\section*{Declaration of competing interest}
The authors declare that they have no known competing financial interests or personal relationships that could have appeared to influence the work reported in this paper.
\section*{Data availability}
No data was used for the research described in the article.
\section*{Acknowledgments}
This paper is supported by the National Natural Science Foundation of China (through grant No. 12171414) and Taishan Scholars Special Project of Shandong Province (through grant No. tsqn202211113). After finishing this manuscript, we learned that Agrahari et al.~\cite[Theorem~7]{AgrahariRC} had independently obtained the Theorem~\ref{thm:path-products-RN}. Although the two proofs both use Hamiltonian paths constructed along coordinate rows, our argument gives an explicit point obtained by averaging the incidence vectors of \(2d\) Hamiltonian paths (see Lemma~\ref{lem:interior-snakes}). Moreover, Theorem~\ref{thm:path-products-RN} further yields the characterization of the RP and SRN cases in Corollary~\ref{cor:RP-SRN-products}. For these reasons, we have retained this part of the proof.
\section*{Declaration of generative AI and AI-assisted technologies in the manuscript preparation process}
During the preparation of this work, the author(s) used ChatGPT 5.6 Pro for checking grammatical errors and typos during the revision stage. The author(s)reviewed andedited the output as needed and take full responsibility for the content of the published article.
\section*{Conflict of interest statement}
The authors declare that they have no conflict of interest.

\end{document}